\documentclass[11pt]{article}

\usepackage{amsmath,amssymb,amsthm}
\usepackage[margin=1in]{geometry}
\usepackage[hidelinks]{hyperref}

\theoremstyle{plain}
\newtheorem{theorem}{Theorem}
\newtheorem{lemma}[theorem]{Lemma}
\newtheorem{corollary}[theorem]{Corollary}
\newtheorem{claim}{Claim}[theorem]

\theoremstyle{remark}

\newcommand{\F}{\mathcal{F}}
\newcommand{\N}{\mathcal{N}}

\newcommand{\bb}{\binom}
\newcommand{\abs}[1]{\left|#1\right|}

\newcommand{\bsupp}{\operatorname{bsupp}}

\title{Refinements of Alon–Babai–Suzuki-type intersection theorems via non-shadows and binomial support}
\author{
Jiangdong Ai\thanks{School of Mathematical Sciences and LPMC, Nankai University. \texttt{jd@nankai.edu.cn}. Funded by the National Natural Science Foundation of China (No.~12522117, No.~12401456), the Natural Science Foundation of Tianjin (No.~24JCQNJC01960), and the Fundamental and Interdisciplinary Disciplines Breakthrough Plan of the Ministry of Education of China (JYB2025XDXM207).}%
\and
Mingyu Liu\thanks{School of Mathematical Sciences and LPMC, Nankai University. \texttt{778016326@qq.com.}
}}
\date{}

\begin{document}
\maketitle

\begin{abstract}
We prove a multilevel \emph{non-shadow} refinement of the Alon--Babai--Suzuki (ABS)
nonuniform restricted-intersection theorem.
Let $K=\{k_1,\dots,k_r\}$ and let $L$ be a set with $|L|=s$.
If $\F\subseteq \bigcup_{k\in K}\binom{[n]}{k}$ is $L$-intersecting and $k_i>s-r$ for every $i$, then
$$
  |\F|+\sum_{j=s-r+1}^{s}|\N_j(\F)|\le N(n,s,r),
  \qquad\text{equivalently}\qquad
  |\F|\le \sum_{j=s-r+1}^{s}|\partial_j\F|.
$$
Thus the ABS bound is sharpened by the total non-shadow deficit on the top $r$ levels.

In the modular setting, we take a coefficient-sensitive viewpoint: the polynomial method depends not just on the degree of the annihilator polynomial
$P_L(t)=\prod_{\ell\in L}(t-\ell)\in\mathbb{F}_p[t]$,
but on which binomial terms actually appear in it.
This yields a gap-free modular bound depending only on the active support levels of $P_L$.
For almost-initial residue patterns
$$
  L=\{0,1,\dots,s-m-1\}\cup R \pmod p
$$
we obtain the collapse
$$
  |\F|\le \sum_{i=0}^{m}\binom{n}{s-i}.
$$
In particular, for consecutive residues $L=\{0,1,\dots,s-1\}\pmod p$ we get the sharp bound
$|\F|\le \binom{n}{s}$, giving a partial negative answer to a question of Alon--Babai--Suzuki:
the modular ABS bound $N(n,s,r)$ is not attainable in the consecutive-residue regime whenever $r\ge 2$.
\end{abstract}

\section{Introduction}

Restricted-intersection problems form one of the central themes of extremal set theory.
Given a set $L\subseteq \mathbb{Z}_{\ge 0}$, a family $\F\subseteq 2^{[n]}$ is called
\emph{$L$-intersecting} if
$$
  \abs{A\cap B}\in L\qquad\text{for all distinct }A,B\in\F.
$$
This framework contains, among others, the Erd\H{o}s--Ko--Rado theorem~\cite{EKR}, Fisher-type inequalities~\cite{Fisher},
and the Ray--Chaudhuri--Wilson theorem~\cite{RCW}.
A major source of progress in the area is the linear algebra method of Frankl--Wilson,
Ray--Chaudhuri--Wilson, Snevily, and others
\cite{FranklW,RCW,Snevily95,Snevily03}.

The nonuniform setting was treated by Alon, Babai, and Suzuki \cite{ABS}.
For integers $n\ge 1$ and $1\le r\le s$, write
\begin{equation}\label{eq:Nnsr}
  N(n,s,r):=\bb{n}{s}+\bb{n}{s-1}+\cdots+\bb{n}{s-r+1}.
\end{equation}

\begin{theorem}[Alon--Babai--Suzuki \cite{ABS}]\label{thm:ABS}
Let $K=\{k_1,\dots,k_r\}$ and $L=\{\ell_1,\dots,\ell_s\}$ be sets of nonnegative integers, and assume
$k_i>s-r$ for every $i$.
If $\F\subseteq \bigcup_{k\in K}\binom{[n]}{k}$ is $L$-intersecting, then
$$
  \abs{\F}\le N(n,s,r).
$$
\end{theorem}

The classical ABS proof constructs a triangular family of multilinear polynomials and compares it with the
space of all multilinear polynomials of degree at most $s$.
The first point of the present paper is that one can often do better by adjoining additional monomials
coming from \emph{non-shadows}.
Indeed, if a $j$-set $T$ is not contained in any member of $\F$, then the monomial $x_T$ vanishes on all of
$\F$ and can be added to the ABS polynomial family without destroying linear independence.
This idea appeared in a different context in the work of Gao--Liu--Xu \cite{GLX};
here we show that in the nonuniform ABS setting it can be exploited simultaneously on several levels.

For $t\in\{0,1,\dots,n\}$, let
$$
  \partial_t\F:=\{T\in\binom{[n]}{t}: \exists F\in\F\text{ with }T\subseteq F\}
$$
be the $t$-shadow of $\F$, and let
$$
  \N_t(\F):=\binom{[n]}{t}\setminus \partial_t\F
$$
be the $t$-non-shadow.
Our first main theorem is the following multilevel refinement of Theorem~\ref{thm:ABS}.

\begin{theorem}[Multilevel non-shadow ABS theorem]\label{thm:main}
Under the hypotheses of Theorem~\ref{thm:ABS},
\begin{equation}\label{eq:main}
  \abs{\F}+\sum_{j=s-r+1}^{s}\abs{\N_j(\F)}\le N(n,s,r).
\end{equation}
Equivalently,
\begin{equation}\label{eq:shadowform}
  \abs{\F}\le \sum_{j=s-r+1}^{s}\abs{\partial_j\F}.
\end{equation}
\end{theorem}
Theorem~\ref{thm:main} is sharp.
Indeed, if $L=\{0,1,\dots,s-1\}$ and
$$
  \F=\bigcup_{k=s-r+1}^{s}\binom{[n]}{k},
$$
then $\partial_j\F=\binom{[n]}{j}$ for every $j\in\{s-r+1,\dots,s\}$, and hence equality holds in
\eqref{eq:shadowform}.

Thus the ABS upper bound is sharpened by the total deficit of the top $r$ shadow levels.
In particular, if $\abs{\F}$ is within $\Delta$ of $N(n,s,r)$, then
$$
  \sum_{j=s-r+1}^{s}\abs{\N_j(\F)}\le \Delta,
$$
so every extremal or near-extremal family must almost cover each of the top $r$ relevant levels.

Our second main contribution concerns the modular setting and is conceptually different.
Let $p$ be a prime and let $L\subseteq\mathbb{F}_p$ with $\abs{L}=s<p$.
Define the annihilator polynomial
$$
  P_L(t):=\prod_{\ell\in L}(t-\ell)\in\mathbb{F}_p[t].
$$
Since $s<p$, we may expand $P_L$ in the binomial basis as
\begin{equation}\label{eq:PL-expand-intro}
  P_L(t)=\sum_{j=0}^{s} c_j(L)\binom{t}{j}.
\end{equation}
We write
$$
  \bsupp(L):=\{j\in\{0,1,\dots,s\}: c_j(L)\neq 0\}
$$
for the \emph{binomial support} of $P_L$.

The binomial basis is natural here because $\binom{|A\cap B|}{j}$ counts the number of $j$-subsets contained in $A\cap B$.
In other words, each nonzero coefficient $c_j(L)$ tells us that the argument uses the $j$th level of the
Boolean lattice.
From this viewpoint, the relevant parameter is not just the degree of the annihilator polynomial, but its
binomial support, which determines the \emph{effective ambient dimension} of the polynomial method.
This leads to a gap-free coefficient-sensitive modular bound in which only the levels in $\bsupp(L)$ matter.
This is complementary to the shifted ABS strengthenings of Chen--Liu and Wang--Wei--Ge
\cite{CL09,WWG18}: rather than imposing separation or gap conditions, our bound depends only on the support
pattern of $P_L$ in \eqref{eq:PL-expand-intro}.

Our coefficient-sensitive theorem implies that every modular family satisfying the usual size and
intersection congruence conditions is bounded by
$$
  \abs{\F}\le \sum_{j\in \bsupp(L)}\binom{n}{j}.
$$
For almost-initial residue patterns
$$
  L=\{0,1,\dots,s-m-1\}\cup R\pmod p,
$$
with $\abs{R}=m$, the support is forced into the top $m+1$ levels, and hence
$$
  \abs{\F}\le \sum_{i=0}^{m}\binom{n}{s-i}.
$$
In particular, for the consecutive pattern $L=\{0,1,\dots,s-1\}\pmod p$ one has
$P_L(t)=(t)_s=s!\binom{t}{s}$, so only the top level survives and we obtain the sharp bound
$\abs{\F}\le\binom{n}{s}$.
Consequently, the modular ABS bound $N(n,s,r)$ is not attainable whenever $r\ge 2$.

\medskip
Our paper is organized as follows.
In Section 2, we prove Theorem~\ref{thm:main} and present a modular non-shadow variant under the stronger low-level exclusion hypothesis $k>s-r$ for all $k\in K$.
In Section~\ref{sec:coeff}, we prove the coefficient-sensitive modular bound and its sparse-level collapse
corollaries.
We conclude with a brief discussion of extremality, sharpness, and possible extensions beyond the Boolean
setting.

\section{The multilevel non-shadow ABS theorem}\label{sec:proof-main}

We work in the multilinear polynomial ring over $\Omega:=\{0,1\}^n$.
Every function $h:\Omega\to \mathbb{R}$ is represented uniquely by a multilinear polynomial in
$\mathbb{R}[x_1,\dots,x_n]$.
For $I\subseteq[n]$, write
$$
  x_I:=\prod_{i\in I}x_i,
$$
with $x_\emptyset=1$.
For a subset $J\subseteq[n]$, let $v_J\in\Omega$ be its characteristic vector and write $h(J):=h(v_J)$.
Then
\begin{equation}\label{eq:indicator}
  x_I(J)=\begin{cases}
    1,& I\subseteq J,\\
    0,& I\nsubseteq J.
  \end{cases}
\end{equation}

We use the following lemma of Alon--Babai--Suzuki; see \cite[Lemma~2.1]{ABS}.

\begin{lemma}[Alon--Babai--Suzuki~\cite{ABS}]\label{lem:ABS21}
Let $f:\Omega\to\mathbb{R}$ be a function such that $f(I)\neq 0$ for every $I\subseteq[n]$ with $\abs{I}\le t$.
Then the family $\{x_I f:\abs{I}\le t\}$ is linearly independent in $\mathbb{R}^{\Omega}$.
\end{lemma}

The next lemma is the mechanism that allows non-shadows to be added on several levels at once.

\begin{lemma}\label{lem:adjoin}
Let $f:\Omega\to\mathbb{R}$ and let $t\ge 0$.
Assume that $f(I)\neq 0$ for every $I\subseteq[n]$ with $\abs{I}\le t$.
Let $\mathcal{J}\subseteq2^{[n]}$ be a family such that $\abs{J}>t$ for all $J\in\mathcal{J}$.
Then
$$
  \{x_I f:\abs{I}\le t\}\cup\{x_J:J\in\mathcal{J}\}
$$
are linearly independent in $\mathbb{R}^{\Omega}$.
\end{lemma}

\begin{proof}
Suppose that
$$
  \sum_{\abs{I}\le t}\alpha_I x_I f+\sum_{J\in\mathcal{J}}\beta_J x_J=0
$$
as functions on $\Omega$.
Evaluating on every subset $S\subseteq[n]$ with $\abs{S}\le t$, the second sum vanishes by
\eqref{eq:indicator}, so $\sum_{\abs{I}\le t}\alpha_I x_I f$ vanishes on all such $S$.
Lemma~\ref{lem:ABS21} therefore implies $\alpha_I=0$ for every $\abs{I}\le t$.

We are left with $\sum_{J\in\mathcal{J}}\beta_J x_J=0$.
Order the sets in $\mathcal{J}$ as $J_1,\dots,J_m$ so that
$\abs{J_1}\le\abs{J_2}\le\cdots\le\abs{J_m}$.
Evaluating at $J_a$ gives
$$
  0=\sum_{b=1}^{m}\beta_{J_b}x_{J_b}(J_a).
$$
If $b>a$ and $x_{J_b}(J_a)=1$, then $J_b\subseteq J_a$, hence $\abs{J_b}\le\abs{J_a}$.
By the chosen order this forces $\abs{J_b}=\abs{J_a}$ and then $J_b=J_a$, impossible.
Thus $x_{J_b}(J_a)=0$ for all $b>a$, while $x_{J_a}(J_a)=1$.
So each equation reduces to $\beta_{J_a}=0$.
\end{proof}
We now prove Theorem~\ref{thm:main} by combining the triangular ABS family with monomials indexed by
missing top-level shadows.
\begin{proof}[Proof of Theorem~\ref{thm:main}]
Write $\F=\{A_1,\dots,A_m\}$ and order the sets so that
$\abs{A_1}\le\abs{A_2}\le\cdots\le\abs{A_m}$.

\medskip
For each $i$, define
\begin{equation}\label{eq:fi}
  f_i(x):=\prod_{\ell\in L:\ \ell<\abs{A_i}}(v_{A_i}\cdot x-\ell),
\end{equation}
where $v_{A_i}\cdot x=\sum_{j=1}^{n}(v_{A_i})_j x_j$.

\begin{claim}\label{clm:triangular}
$f_i(A_i)\neq 0$ and $f_i(A_j)=0$ for all $j<i$.
\end{claim}

\begin{proof}
We have $f_i(A_i)=\prod_{\ell\in L,\,\ell<\abs{A_i}}(\abs{A_i}-\ell)\neq 0$
since each factor is nonzero.
If $j<i$, then $\abs{A_j}\le\abs{A_i}$ and, since $\F$ is $L$-intersecting,
$\abs{A_i\cap A_j}\in L$.
Because $A_i\neq A_j$, we have $\abs{A_i\cap A_j}<\abs{A_i}$,
so the factor indexed by $\ell^*:=\abs{A_i\cap A_j}$ in~\eqref{eq:fi} vanishes at $A_j$.
\end{proof}

\noindent Claim~\ref{clm:triangular} implies that $f_1,\dots,f_m$ are linearly independent.

\medskip
Define
$$
  g(x):=\prod_{k\in K}\Bigl(\sum_{j=1}^{n}x_j-k\Bigr).
$$

\begin{claim}\label{clm:filter}
$g(A)=0$ for every $A\in\F$, and $g(I)\neq 0$ for every $I\subseteq[n]$ with $\abs{I}\le s-r$.
Consequently, $\{x_I g:\abs{I}\le s-r\}$ is linearly independent.
\end{claim}

\begin{proof}
For $A\in\F$ we have $\abs{A}\in K$, so some factor of $g$ vanishes at $A$.
If $\abs{I}\le s-r$, then $\abs{I}<k$ for every $k\in K$ (since each $k>s-r$),
so every factor of $g$ is nonzero at $I$, giving $g(I)\neq 0$.
Linear independence then follows from Lemma~\ref{lem:ABS21}.
\end{proof}

\medskip
Let
$$
  \mathcal{J}:=\bigcup_{j=s-r+1}^{s}\N_j(\F),
$$
and set
$$
  \mathcal{B}:=\{f_i:1\le i\le m\}\cup\{x_I g:\abs{I}\le s-r\}\cup\{x_J:J\in\mathcal{J}\}.
$$

\begin{claim}\label{clm:lindep}
$\mathcal{B}$ is linearly independent.
\end{claim}

\begin{proof}
Suppose that
$$
  \sum_{i=1}^{m}\lambda_i f_i
  +\sum_{\abs{I}\le s-r}\mu_I x_I g
  +\sum_{J\in\mathcal{J}}\nu_J x_J=0.
$$
We first show all $\lambda_i$ vanish.
If some $\lambda_i\neq 0$, let $i_0$ be the smallest such index.
Evaluating at $A_{i_0}$: the middle sum vanishes by Claim~\ref{clm:filter} since $g(A_{i_0})=0$;
the right sum vanishes because $x_J(A_{i_0})=0$ for every $J\in\mathcal{J}$
(as $J\nsubseteq A_{i_0}$ by definition of the non-shadow);
and all $f_i(A_{i_0})$ with $i>i_0$ vanish by Claim~\ref{clm:triangular}.
Hence $\lambda_{i_0}f_{i_0}(A_{i_0})=0$, contradicting $f_{i_0}(A_{i_0})\neq 0$.
Thus $\lambda_i=0$ for all $i$.

The remaining relation $\sum_{\abs{I}\le s-r}\mu_I x_I g+\sum_{J\in\mathcal{J}}\nu_J x_J=0$,
together with the fact that $\abs{J}>s-r$ for every $J\in\mathcal{J}$ and that $g(I)\neq 0$
for $\abs{I}\le s-r$ (Claim~\ref{clm:filter}), allows us to apply Lemma~\ref{lem:adjoin}
with $t=s-r$ and $f=g$, which forces all $\mu_I$ and $\nu_J$ to vanish.
\end{proof}

\medskip
Every element of $\mathcal{B}$ is a multilinear polynomial of degree at most $s$
(each $f_i$ has degree at most $s$; each $x_Ig$ has degree $\abs{I}+r\le s$;
each $x_J$ has degree $\abs{J}\le s$).
By Claim~\ref{clm:lindep},
$$
  \abs{\mathcal{B}}\le \dim\operatorname{span}\{x_I:\abs{I}\le s\}=\sum_{i=0}^{s}\binom{n}{i}.
$$
On the other hand,
$$
  \abs{\mathcal{B}}=\abs{\F}+\sum_{i=0}^{s-r}\binom{n}{i}+\sum_{j=s-r+1}^{s}\abs{\N_j(\F)}.
$$
Subtracting $\sum_{i=0}^{s-r}\binom{n}{i}$ from both sides and using
$N(n,s,r)=\sum_{i=s-r+1}^{s}\binom{n}{i}$ gives~\eqref{eq:main}.
Finally, \eqref{eq:shadowform} follows from
$\abs{\N_j(\F)}=\binom{n}{j}-\abs{\partial_j\F}$.
\end{proof}

The same argument also gives a modular non-shadow variant under the low-level exclusion hypothesis
$k>s-r$ for all $k\in K$. This is stronger than the original ABS gap-condition assumption, but it
matches the standing hypothesis in Theorem~\ref{thm:main}. We present it for comparison, because in the next section we will prove a different modular theorem.

\begin{theorem}
Let $p$ be a prime, and let $K,L \subseteq \{0,1,\ldots,p-1\}$ be disjoint with $|K|=r$ and
$|L|=s$. Assume that $k>s-r\geq 0$ for every $k\in K$. If $F$ is a family of subsets of $[n]$ such that
\begin{enumerate}
\item[(i)] $|F| \in K+p\mathbb Z$ for every $F\in F$;
\item[(ii)] $|E\cap F| \in L+p\mathbb Z$ for all distinct $E,F\in F$,
\end{enumerate}
then
$$
|F|+\sum_{j=s-r+1}^{s}|N_j(F)| \le N(n,s,r).
$$
Equivalently,
$$
|F| \le \sum_{j=s-r+1}^{s} |\partial_j F|.
$$
\end{theorem}

\begin{proof}
Write $F=\{A_1,\ldots,A_m\}$ and order the sets so that
$$
|A_1|\le |A_2|\le \cdots \le |A_m|.
$$
For each $i$, define
$$
f_i(x):=\prod_{\ell\in L:\,\ell<|A_i|}(v_{A_i}\cdot x-\ell),
$$
where
$$
v_{A_i}\cdot x=\sum_{j=1}^n (v_{A_i})_j x_j.
$$

We claim that $f_i(A_i)\neq 0$ and $f_i(A_j)=0$ for all $j<i$. Indeed, if $\ell\in L$ and
$\ell<|A_i|$, then $|A_i|-\ell\neq 0$ in $\mathbb F_p$ because $|A_i| \pmod p \in K$ and
$K\cap L=\varnothing$. Hence $f_i(A_i)\neq 0$.

Now let $j<i$. Then $|A_j|\le |A_i|$ and $A_i\neq A_j$, so
$$
|A_i\cap A_j|<|A_i|.
$$
By (ii), there exists $\ell^*\in L$ such that
$$
|A_i\cap A_j|\equiv \ell^* \pmod p.
$$
Writing
$$
|A_i\cap A_j|=qp+\ell^*
$$
with $0\le \ell^*\le p-1$, we have
$$
\ell^*\le |A_i\cap A_j|<|A_i|.
$$
Therefore the factor indexed by $\ell^*$ appears in the product defining $f_i$, and it vanishes at
$A_j$. Thus $f_i(A_j)=0$ for every $j<i$. It follows that $f_1,\ldots,f_m$ are linearly
independent.

Define
$$
g(x):=\prod_{k\in K}\left(\sum_{j=1}^n x_j-k\right).
$$
If $A\in F$, then $|A|\pmod p\in K$, so $g(A)=0$. If $I\subseteq [n]$ satisfies $|I|\le s-r$, then
$|I|<k$ for every $k\in K$, hence $|I|-k\neq 0$ in $\mathbb F_p$ for every $k\in K$, and so
$g(I)\neq 0$. Lemma 3 therefore implies that
$$
\{x_Ig:\ |I|\le s-r\}
$$
is linearly independent over $\mathbb F_p$.

Let
$$
\mathcal J:=\bigcup_{j=s-r+1}^{s} N_j(F),
$$
and set
$$
B:=\{f_i:\ 1\le i\le m\}\cup \{x_Ig:\ |I|\le s-r\}\cup \{x_T:\ T\in\mathcal J\}.
$$
We claim that $B$ is linearly independent. Suppose that
$$
\sum_{i=1}^m \lambda_i f_i+\sum_{|I|\le s-r}\mu_Ix_Ig+\sum_{T\in\mathcal J}\nu_Tx_T=0.
$$
We first show that all $\lambda_i$ vanish. If some $\lambda_i\neq 0$, let $i_0$ be the smallest such
index. Evaluating at $A_{i_0}$, the middle sum vanishes because $g(A_{i_0})=0$, and the last sum
vanishes because $x_T(A_{i_0})=0$ for every $T\in\mathcal J$, since $T$ is contained in no member
of $F$. Also, $f_i(A_{i_0})=0$ for all $i>i_0$ by the triangular property proved above. Hence
$$
\lambda_{i_0}f_{i_0}(A_{i_0})=0,
$$
contradicting $f_{i_0}(A_{i_0})\neq 0$. Thus $\lambda_i=0$ for all $i$.

The remaining relation
$$
\sum_{|I|\le s-r}\mu_Ix_Ig+\sum_{T\in\mathcal J}\nu_Tx_T=0
$$
together with the facts that $|T|>s-r$ for every $T\in\mathcal J$ and that $g(I)\neq 0$ for all
$|I|\le s-r$ allows us to apply Lemma 4 with $t=s-r$ and $f=g$. Therefore all $\mu_I$ and
$\nu_T$ vanish, and $B$ is linearly independent.

Every element of $B$ is a multilinear polynomial of degree at most $s$: each $f_i$ has degree at
most $s$, each $x_Ig$ has degree $|I|+r\le s$, and each $x_T$ has degree $|T|\le s$. Hence
$$
|B|\le \dim\operatorname{span}\{x_I:\ |I|\le s\}
 =\sum_{i=0}^{s}\binom{n}{i}.
$$
On the other hand,
$$
|B|=|F|+\sum_{i=0}^{s-r}\binom{n}{i}+\sum_{j=s-r+1}^{s}|N_j(F)|.
$$
Subtracting $\sum_{i=0}^{s-r}\binom{n}{i}$ from both sides and using
$$
N(n,s,r)=\sum_{i=s-r+1}^{s}\binom{n}{i}
$$
gives
$$
|F|+\sum_{j=s-r+1}^{s}|N_j(F)|\le N(n,s,r).
$$
The equivalent form follows from
$$
|N_j(F)|=\binom{n}{j}-|\partial_jF|.
$$
\end{proof}

\section{Coefficient-sensitive modular bounds}\label{sec:coeff}

In the modular setting, the annihilator polynomial may have substantially smaller binomial support than the
full set $\{0,1,\dots,s\}$.
As a result, the effective ambient dimension of the polynomial method can collapse to a much smaller sum of
Boolean levels.
This yields a gap-free bound depending only on $\bsupp(L)$.

For $0\le j\le s$, let $\binom{t}{j}\in\mathbb{F}_p[t]$ denote the binomial polynomial.
Since $s<p$, the polynomials $\binom{t}{0},\dots,\binom{t}{s}$ form a basis of the space of polynomials of
degree at most $s$ over $\mathbb{F}_p$.
Given $L\subseteq\mathbb{F}_p$ with $\abs{L}=s$, write
\begin{equation}\label{eq:PL-expand}
  P_L(t)=\sum_{j=0}^{s} c_j(L)\binom{t}{j},
\end{equation}
and recall that $\bsupp(L)=\{j:c_j(L)\neq 0\}$.

The key point is that the binomial basis is combinatorially canonical:
$\binom{\abs{A\cap B}}{j}$ counts the number of $j$-subsets contained in $A\cap B$.
Thus each active coefficient $c_j(L)$ corresponds to an actual Boolean level, and only those levels in
$\bsupp(L)$ contribute to the ambient space in the polynomial method.

Our first result in this direction is the following support-sensitive modular bound.
\begin{theorem}\label{thm:coeffbound}
Let $p$ be a prime and let $1\le s\le p-1$.
Let $L\subseteq \mathbb{F}_p$ with $\abs{L}=s$, and let $K\subseteq \mathbb{F}_p\setminus L$.
Let $\F\subseteq 2^{[n]}$ satisfy
\begin{enumerate}
  \item[(i)] $\abs{F}\pmod p\in K$ for every $F\in\F$;
  \item[(ii)] $\abs{E\cap F}\pmod p\in L$ for all distinct $E,F\in\F$.
\end{enumerate}
Then
$$
  \abs{\F}\le \sum_{j\in \bsupp(L)}\binom{n}{j}.
$$
\end{theorem}

\begin{proof}
For each $j$ and each $A\subseteq[n]$, let $u_A^{(j)}\in\mathbb{F}_p^{\binom{[n]}{j}}$ be the $j$-incidence vector
whose $S$-coordinate is $1$ if $S\subseteq A$ and $0$ otherwise.
Then for any $A,B\subseteq[n]$,
$$
  \langle u_A^{(j)},u_B^{(j)}\rangle=\binom{\abs{A\cap B}}{j}.
$$
Let
$$
  V:=\bigoplus_{j\in\bsupp(L)}\mathbb{F}_p^{\binom{[n]}{j}}
$$
and define
$$
  w_A:=\bigoplus_{j\in\bsupp(L)}u_A^{(j)}\in V.
$$
Consider the symmetric bilinear form $B:V\times V\to\mathbb{F}_p$ given by
$$
  B(x,y):=\sum_{j\in\bsupp(L)} c_j(L)\langle x^{(j)},y^{(j)}\rangle,
$$
where $x^{(j)}$ denotes the $j$th component of $x$.
Then for any $A,B\subseteq[n]$,
$$
  B(w_A,w_B)=\sum_{j=0}^{s}c_j(L)\binom{\abs{A\cap B}}{j}=P_L(\abs{A\cap B}).
$$
If $A\neq B$ are in $\F$, then $\abs{A\cap B}\pmod p\in L$, so $P_L(\abs{A\cap B})=0$.
If $A\in\F$, then $\abs{A}\pmod p\in K\subseteq\mathbb{F}_p\setminus L$, so $P_L(\abs{A})\neq 0$.
Hence the vectors $\{w_A:A\in\F\}$ are pairwise $B$-orthogonal with nonzero self-pairing, and are therefore
linearly independent.
Consequently,
$$
  \abs{\F}\le \dim V=\sum_{j\in\bsupp(L)}\binom{n}{j}.
$$
\end{proof}

The same support-sensitive argument also admits non-shadows on the active levels.

\begin{theorem}\label{thm:coeffbound-nonshadow}
Under the hypotheses of Theorem~\ref{thm:coeffbound},
$$
  \abs{\F}+\sum_{j\in\bsupp(L)}\abs{\N_j(\F)}\le \sum_{j\in\bsupp(L)}\binom{n}{j}.
$$
Equivalently,
$$
  \abs{\F}\le \sum_{j\in\bsupp(L)}\abs{\partial_j\F}.
$$
\end{theorem}

\begin{proof}
Use the same ambient space $V$ and vectors $w_A$ as in the proof of Theorem~\ref{thm:coeffbound}.
For each $j\in\bsupp(L)$ and each $T\in\N_j(\F)$, let $e_T^{(j)}\in V$ be the standard basis vector at the
coordinate $T$ in the $j$th component.
Since $T$ is not contained in any member of $\F$, every $w_A$ has zero $T$-coordinate, while
$e_T^{(j)}$ has $T$-coordinate equal to $1$.
Therefore the family
$$
  \{w_A:A\in\F\}\cup \bigcup_{j\in\bsupp(L)}\{e_T^{(j)}:T\in\N_j(\F)\}
$$
is linearly independent. Indeed, looking at the coordinate corresponding to a fixed missing $T$ first forces the coefficient of
$e_T^{(j)}$ to vanish, since no vector $w_A$ has a nonzero entry there.
After removing all such terms, any remaining dependence would be a dependence among the vectors
$\{w_A:A\in\F\}$, which is impossible by the orthogonality argument in the proof of
Theorem~\ref{thm:coeffbound}.
A dimension count now yields the desired inequality.
\end{proof}

The most useful structural consequence arises when the residue set contains a long initial segment.
In that case the support is automatically forced into the top few levels.

\begin{corollary}\label{cor:almost-initial}
Let $p$ be a prime and let $1\le s\le p-1$.
Fix $m$ with $0\le m\le s-1$, and let $R\subseteq\{0,1,\dots,p-1\}$ be a set of size $m$
disjoint from $\{0,1,\dots,s-m-1\}$.
Set
$$
  L:=\{0,1,\dots,s-m-1\}\cup R\subseteq\mathbb{F}_p,
$$
and let $K\subseteq\mathbb{F}_p\setminus L$.
If $\F\subseteq 2^{[n]}$ satisfies the modular conditions \textup{(i)}--\textup{(ii)} of Theorem~\ref{thm:coeffbound}, then
$$
  \abs{\F}\le \sum_{i=0}^{m}\binom{n}{s-i}.
$$
More precisely,
$$
  \abs{\F}+\sum_{j\in\bsupp(L)}\abs{\N_j(\F)}\le \sum_{i=0}^{m}\binom{n}{s-i}.
$$
\end{corollary}

\begin{proof}
Write $(t)_j:=t(t-1)\cdots(t-j+1)$.
For the present $L$,
$$
  P_L(t)=(t)_{s-m}\prod_{\rho\in R}(t-\rho).
$$
Multiplying $(t)_{s-m}$ by a polynomial of degree at most $m$ cannot create falling-factorial terms of degree
below $s-m$.
Indeed, using the identity
$$
  t\,(t)_k=(t)_{k+1}+k\,(t)_k,
$$
and inducting on the degree of the multiplier, one sees that
$$
  P_L(t)\in \operatorname{span}\{(t)_s,(t)_{s-1},\dots,(t)_{s-m}\}.
$$
Since $(t)_j=j!\binom{t}{j}$ and $j!\neq 0$ in $\mathbb{F}_p$ for $j\le s<p$, this implies
$$
  \bsupp(L)\subseteq\{s-m,s-m+1,\dots,s\}.
$$
The stated bounds now follow immediately from Theorem~\ref{thm:coeffbound-nonshadow}.
\end{proof}

\begin{corollary}\label{cor:consecutive}
Let $p$ be a prime and let $2\le s\le p-1$.
If
$$
  L=\{0,1,\dots,s-1\}\subseteq\mathbb{F}_p,
$$
and $K\subseteq\mathbb{F}_p\setminus L$, then every family $\F\subseteq 2^{[n]}$ satisfying the modular
conditions \textup{(i)}--\textup{(ii)} of Theorem~\ref{thm:coeffbound} obeys
$$
  \abs{\F}\le \binom{n}{s}.
$$
Moreover, if $s\in K$, equality holds for $\F=\binom{[n]}{s}$.
In particular, whenever $|K|=r$ with $2 \le r \le s$, and $n \ge s$, the modular ABS bound $N(n,s,r)$ is not attainable.
\end{corollary}

\begin{proof}
Here $P_L(t)=(t)_s=s!\binom{t}{s}$, so $\bsupp(L)=\{s\}$.
The bound follows from Theorem~\ref{thm:coeffbound} or Theorem~\ref{thm:coeffbound-nonshadow}.
If $s\in K$, then $\binom{[n]}{s}$ is an admissible family of size $\binom{n}{s}$.
Since
$$
  N(n,s,r)=\binom{n}{s}+\binom{n}{s-1}+\cdots+\binom{n}{s-r+1}>\binom{n}{s}
$$
for every $2 \le r \le s$, the final claim is immediate.
\end{proof}

\section{Concluding remarks}

Theorem~\ref{thm:main} shows that the ABS extremal problem is controlled not only by the upper bound
$N(n,s,r)$, but also by how completely $\F$ covers the top $r$ levels of the Boolean lattice.
In particular, equality forces
$$
  \N_j(\F)=\emptyset \qquad \text{for all } j\in\{s-r+1,\dots,s\},
$$
so every extremal family must cover every set on each relevant top level.
Finally, the binomial-support viewpoint is not confined to the modular setting. Let
\(L \subseteq \mathbb Z_{\ge 0}\) with \(|L|=s\), and expand
$$
P_L(t):=\prod_{\ell\in L}(t-\ell)=\sum_{j=0}^{s} c_j(L)\binom{t}{j}
$$
in \(\mathbb R[t]\). Write
$$
bsupp(L):=\{j\in\{0,1,\ldots,s\}: c_j(L)\neq 0\}.
$$
Then the same orthogonality argument as in Theorems~6 and~7 yields a nonmodular
support-sensitive analogue in which the hypothesis \(k_i>s-r\) is replaced simply by
\(K\cap L=\varnothing\).

\begin{theorem}
Let \(K=\{k_1,\ldots,k_r\}\) and \(L=\{\ell_1,\ldots,\ell_s\}\) be sets of nonnegative integers with
\(K\cap L=\varnothing\). If \(F \subseteq \bigcup_{k\in K}\binom{[n]}{k}\) is \(L\)-intersecting, then
$$
|F|+\sum_{j\in bsupp(L)} |N_j(F)|
\le
\sum_{j\in bsupp(L)} \binom{n}{j}.
$$
Equivalently,
$$
|F|\le \sum_{j\in bsupp(L)} |\partial_jF|.
$$
In particular,
$$
|F|\le \sum_{j\in bsupp(L)} \binom{n}{j}.
$$
\end{theorem}

We omit the proof.
In the modular setting, Theorems~\ref{thm:coeffbound} and~\ref{thm:coeffbound-nonshadow} suggest the following natural problem:
for a given residue pattern $L$, when is the bound
$$\sum_{j\in\bsupp(L)}\binom{n}{j}$$
best possible? And how do the equality or near-equality cases reflect the active shadow levels determined by
$\bsupp(L)$?
Even for simple nonconsecutive residue sets this appears to be delicate.

It is natural to ask whether the multilevel non-shadow mechanism extends beyond the Boolean lattice
to the broader semilattice framework of Alon--Babai--Suzuki.
Such an extension would in particular yield a Grassmann-lattice analogue.
We leave this direction to future work and keep the present paper focused on the Boolean
and modular settings.

More broadly, the results of this paper indicate that in ABS-type intersection problems the correct measure of
the polynomial method is often not the full degree-$s$ ambient space, but the smaller effective ambient space
selected either by missing shadows or by the binomial support of the annihilator polynomial.

\end{document}